\documentclass[12pt,reqno,oneside]{amsart}
\usepackage{amsmath,amsthm,amsfonts,amssymb}
\usepackage{indentfirst}
\usepackage{url}
\usepackage{hyperref}
\usepackage{graphicx}
\usepackage{color}

\theoremstyle{plain}
\newtheorem{theorem}{Theorem}[section]

\newtheorem{lem}[theorem]{Lemma}
\newtheorem{prop}[theorem]{Proposition}

\theoremstyle{definition}
\newtheorem{defn}[theorem]{Definition}

\newtheorem{obs}[theorem]{Remark}

\numberwithin{equation}{section}
\numberwithin{figure}{section}

\newcommand{\ccp}{{colonization and collapse process}}

\newcommand{\la}{\lambda}
\newcommand{\ee}{\epsilon}

\begin{document}

\baselineskip=18pt

\title{Colonization and Collapse}

\author[F\'abio P. Machado]{F\'abio~Prates~Machado}
\address[F\'abio P. Machado]{Statistics Department, Institute of Mathematics and Statistics, University of S\~ao Paulo, CEP 05508-090, S\~ao Paulo, SP, Brazil.}
\email{fmachado@ime.usp.br}
\thanks{Research supported by CNPq (141046/2013-9 and 310829/2014-3) and FAPESP (09/52379-8).}

\author[Alejandro Rold\'an]{Alejandro~Rold\'an-Correa}
\address[Alejandro Rold\'an]{Instituto de Matem\'aticas, Universidad de Antioquia, Calle 67, no 53-108, Medellin, Colombia}
\email{alejandro.roldan@udea.edu.co}

\author[Rinaldo B. Schinazi]{Rinaldo B. Schinazi}
\address[Rinaldo B. Schinazi]{Math Department, Institute of Mathematics, University of Colorado at Colorado Springs, 80933-7150, Colorado Springs, CO, USA.}
\email{rschinaz@uccs.edu}

\date{\today}

\begin{abstract} 
Many species live in colonies that thrive for a while and then 
collapse. Upon collapse very few individuals survive. The survivors start new 
colonies at other sites that thrive until they collapse, and so on. We introduce spatial 
and non-spatial stochastic processes for modeling such population 
dynamic. Besides testing whether dispersion helps survival in a model experiencing large 
fluctuations, we obtain conditions for the population to get extinct or to survive.
\end{abstract}

\maketitle

\section{Introduction}
\label{S: Introduction}

A metapopulation model refers to populations that  are spatially structured into assemblages of 
local populations that are connected via migrations. Each local population evolves without spatial 
structure;  it can increase or decrease, survive, get extinct or migrate in different 
ways. Many biological phenomena may influence the dynamics of a metapopulation; species  adopt 
different strategies to increase its survival probability. See Hanski~\cite{Hanski} for more about metapopulations.

Some metapopulations  (such as ants) live in colonies that thrive for a while and then collapse. Upon collapse very few
individuals survive. The survivors start new colonies at other
vertices that thrive until they collapse, and so on.
In this paper, we introduce stochastic models to model this population dynamic and to test whether dispersion helps survival.
Our non-spatial stochastic models are reminiscent of catastrophe models, see Brockwell~\cite{Brockwell}. However, instead of having the whole population living in a single colony we now have the population dispersed in a random number of colonies. We show that dispersion of the population helps survival, see Section 4. Our spatial model is similar to a contact process, see Liggett~\cite{Liggett}, in that individuals give birth to new individuals in neighboring sites. However, in our model there is no limit in the number of individuals per site. Moreover, at collapse time a site loses all its individuals at once. This introduces large fluctuations that are non standard in interacting particle systems and complicate the analysis, see Section 5.

This paper is divided into five sections. In Section 2 we define a spatial stochastic process for colonization and collapse and present some of its properties. In Section 3 the main results are established. In Section 4 we introduce a non-spatial version of our model and compare it to other models known in the literature. Finally, in Section 5 we prove the results stated in Section 3.

\section{Spatial model}
We denote by $\mathcal{G} = (V, E)$ a connected non-oriented graph of locally bounded degree, 
where $V: = V (\mathcal{G})$ is the set of vertices of $\mathcal{G}$, and 
$E: = E (\mathcal{G})$ is the set of edges of  of $\mathcal{G}$. Vertices are considered 
neighbors if they belong to a common edge. The \textit{degree} of a vertex $x \in V$ is the number 
of edges that have $x$ as an endpoint. A graph is \textit{locally bounded} if all its vertices have finite degree. A graph is $k-$\textit{regular} if all its vertices have degree $k.$
The distance $d(x, y)$ between vertices $x$ and $y$ is the minimal amount of edges that one must 
pass in order to go from $x$ to $y.$ By $\mathbb{Z}^d$ we denote the graph whose set of vertices is
$\mathbb{Z}^d$ and the set of edges is  
$\{\left\langle (x_1,\ldots,x_d),(y_1,\ldots,y_d)\right\rangle: |x_1-y_1|+\ldots|x_d-y_d|=1\}.$ 
Besides, by $\mathbb{T}^d$, $d\geq2$, we denote the degree $d+1$ homogeneous tree. \\

At any time each vertex of ${\mathcal G}$ may be either occupied by a colony or 
empty. Each colony is started by a single individual. The number of individuals in 
each colony behaves as a Yule process (i.e. pure birth)  with birth rate $\lambda\geq 0$. To each 
vertex is associated a Poisson process with rate 1 in such a way that when the exponential 
time occurs at a vertex occupied by a colony, that colony collapses and its vertex becomes empty. At the time 
of collapse each individual in the colony survives with a (presumably small) probability 
$p \in (0,1)$ or dies with probability $1-p$. Each individual that survives tries to found 
a new colony on one of the nearest neighbor vertices by first picking a vertex at random. 
If the chosen vertex is occupied, that individual dies, otherwise the individual founds there a 
new colony. We denote by 
$CC(\mathcal{G},\lambda,p)$ the Colonization and Collapse model.\\

The $CC(\mathcal{G},\lambda,p)$ is a continuous time Markov process whose state space is $\mathbb{N}^V$ and whose evolution (status at time $t$) is denoted by $\eta_t$. For a vertex $x \in V$, $\eta_t(x)=i$ means that at the time $t$ there are $i$ 
individuals at the vertex $x$. We consider $|\eta_t| = \sum_{x \in V} \eta_t(x)$.

\begin{defn} 
Let $\eta_t$ be a $CC(\mathcal{G},\lambda,p)$, starting with a finite number of colonies. If  $\mathbb{P}(|\eta_t|\geq1 \text{ for all } t\geq0)>0 $ we say that  $\eta_t$ {\it survives (globally)}. Otherwise, we say that $\eta_t$ {\it dies out (globally)}.
\end{defn}

\begin{obs}
If the process $\eta_t$ starts from an infinite number of colonies, then 
$\mathbb{P}(|\eta_t|\geq1 \text{ for all } t\geq0)=1,$ which means that 
$\eta_t$ survives with probability 1. Still we can see local death according 
to the following definition.
\end{obs}

\begin{defn} Let $\eta_t$ be a $CC(\mathcal{G},\lambda,p).$ We say that  $\eta_t$ {\it dies locally} if for any vertex   $x \in V$ there is a finite random time $T$ such that $\eta_t(x)=0$ for  all $t>T$. Otherwise we say that $\eta_t$ {\it survives locally}.
\end{defn}

\begin{obs} Local death corresponds to a finite number of colonizations for every vertex. 
It is clear that global death implies local death but the opposite is not always truth. As an
example consider $\eta_t$ a $CC(\mathbb{Z}^3,0,1)$ with $|\eta_0 |=1$. In this case $\eta_t$ can be seen as a symmetric
random walk on $\mathbb{Z}^3$, therefore transient, which implies that $\eta_t$ dies locally but
survives globally.
\end{obs}

By coupling arguments one can see that $\mathbb{P}(|\eta_t|\geq1 \text{ for all } t\geq0)$ is a non-decreasing function of $\lambda$ and also of $p$. So we define  
$$\lambda_c(p, \mathcal{G}):=\inf\{\lambda: \mathbb{P}^{\delta_x}(|\eta_t|\geq1 \text{ for all } t\geq0) >0 \},$$
where $x$ is a fixed vertex, and  $\mathbb{P}^{\delta_x}$ is the law of the process $\eta_t$ starting with one colony at $x$. The function $\lambda_c(p,\mathcal{G})$ is non-increasing on $p$. Moreover, $\lambda_c(1,\mathcal{G})=0$ and $\lambda_c(0,\mathcal{G})=\infty$.

\begin{defn}
Let $\eta_t$ be a $CC(\mathcal{G},\lambda,p)$ with $0<p<1$. We say that $\eta_t$ exhibits \textit{phase transition} (on $\lambda$) if $0<\lambda_c(p,\mathcal{G})<\infty.$ 
\end{defn}

\begin{obs} Using coupling arguments, we can construct $\eta_t$ and $\hat{\eta}_t$ be two copies of $CC(\mathcal{G},\lambda,p)$ such that $\eta_t \leq \hat{\eta}_t$ for all times $t>0$, provided that $\eta_0 \leq \hat{\eta}_0$. This monotonic property implies that if $\eta_t$ survives, $\hat{\eta}_t$ also does.  
Moreover, if $\hat{\eta}_t$ dies out (or dies locally) then $\eta_t$ does too.
\end{obs}

Observe that, as the number of individuals per vertex is not bounded, it is conceivable that the process survives on a finite graph. Next we show that it does not happen.

\begin{prop}
\label{P: ExtinctionFiniteVolume}
For any finite graph and starting from any initial configuratiion, the colonization and collapse process
dies out.
\end{prop}

\begin{proof}

Let $\alpha$ be the probability at a colony collapse time, zero
individuals attempts to found new colonies at neighboring vertices.

From the fact that the probability that a Yule process, starting from one individual, has $j$ individuals at time $t$ is
$e^{-\lambda t}(1-e^{-\lambda t})^{j-1},$ we have that 
\[ \alpha=\int_0^\infty e^{-t}\sum_{j= 1}^\infty e^{-\lambda t}(1-e^{-\lambda t})^{j-1}(1-p)^{j}dt. \]

Let $n$ and $m$ be the number of vertices of $\mathcal{G}$ and its maximum degree, respectively. 
In order to
show sufficient conditions for extinction we couple the number of colonies in the original 
model to $X_t$, the following continuous time branching process. 
Each individual is independently associated to an exponential random variable of rate 1 in such a way that, when its exponential time occurs, it dies with probability $\alpha$ or
is replaced by $m$ individuals with probability $1-\alpha$. We also consider a restriction
that makes the total number of individuals always smaller or equal than $n$ by suppressing 
the births that would make $X_t$ larger than $n$.

Let $C_t$ be the number of colonies at time $t$ in the colonization and collapse process.
At time $t=0$ let $X_0=C_0$. Moreover, we couple each individual in $X_0$ to a colony in
the colonization and collapse process by using the same exponential random variable of rate 1.
When an exponential occurs there are two possibilities. With probability $\alpha$ both
$X_t$ and $C_t$  decrease by 1. With probability $1-\alpha$ the process $X_t$ grows by $m-1$ individuals and $C_t$ grows by at most $m-1$ colonies. This is so because in the colonization 
process we have spatial constraints and attempted colonizations only occur at vertices that are empty.
Hence, new colonies correspond to births for $X_t$. We couple each new colony to a new
individual in $X_t$ by using the same mean 1 exponential random variable. This coupling yields
for all $t \geq 0 $
\[ X_t \geq C_t. \]
Note now that $\alpha>0$ and that $X_t$ is a finite Markov process with 
an absorbing state. Hence, $X_t$ dies out with probability 1. So does $C_t$ and therefore the 
\ccp.
\end{proof}

\section{Main Results}

Next, we show sufficient conditions for global extinction and local extinction for the \ccp\ 
on infinite graphs.

\begin{theorem}\label{extinction} Let $\mathcal{G}$ be a $m$-regular graph, $\eta_t$ a $CC(\mathcal{G},\lambda,p)$ and $$\mu(m) = m-\frac{m}{\lambda}\sum_{k\geq 1} B\left(1+\frac{1}{\lambda},k\right)\left(1-\frac{p}{m}\right)^k,$$ where $B(a,b)=\int_0^1 u^{a-1}(1-u)^{b-1}du$ is the beta function.  
\begin{itemize}
\item[$(i)$] If $\mu(m)\leq1$ and $|\eta_0| < \infty$, then $\eta_t$ dies out locally and globally.
\item[$(ii)$] Let $\mathcal{G}=\mathbb{Z}^d, \ d \ge 1.$ If $\mu(2d)<1$ and $\eta_0(x)\leq 1$ for every $x$ in  $\mathbb{Z}^d$ then $\eta_t$ dies locally.  
\item[$(iii)$] Let $\mathcal{G}=\mathbb{T}^d, \ d \ge 1.$  If $\mu(d+1)<1/d$ and $\eta_0(x)\leq 1$ for every $x$ in  $\mathbb{T}^d$ then $\eta_t$ dies locally.  
\end{itemize}
\end{theorem}

\begin{obs} Observe that for all $m\geq 1$ and $\lambda$ fixed, there exists $p>0$ such that $\mu(m)\leq1$. Furthermore, $\mu(m)$ can be expressed in terms of the Gauss hypergeometric function ${_2F_1}$ (see
Luke~\cite{Luke}), 

$$\mu(m)=m-\frac{m-p}{\lambda+1}{_2F_1}\left(1,1;2+\frac{1}{\lambda};1-\frac{p}{m}\right).$$    
\end{obs}

Next, we show sufficient conditions for survival for \ccp\ 
on some infinite graphs.

\begin{theorem}
\label{T: survival} 
For $p>0$ and $\lambda:=\lambda(p,\mathcal{G})>0$ large enough, the $CC(\mathcal{G},\lambda,p)$ with $\mathcal{G}=\mathbb{Z}^d$ or $\mathbb{T}^d$ survives globally and locally.
\end{theorem}

\begin{obs} From Theorems \ref{extinction} and \ref{T: survival}  it follows that
for $\mathcal{G}=\mathbb{Z}^d$ or $\mathbb{T}^d$ and $p\in (0,1)$, there exists phase
transition (on $\lambda$) for $CC(\mathcal{G},\lambda,p).$ So, there exists a function
$\lambda_c(\cdot,\mathcal{G}):(0,1)\rightarrow \mathbb{R}^+$ such that the survival and extinction regime
for $CC(\mathcal{G},\lambda,p)$ can be represented as in Figure~\ref{F: TransicaoAnt}.
\end{obs}

\begin{figure}[ht]
\includegraphics[width=6cm]{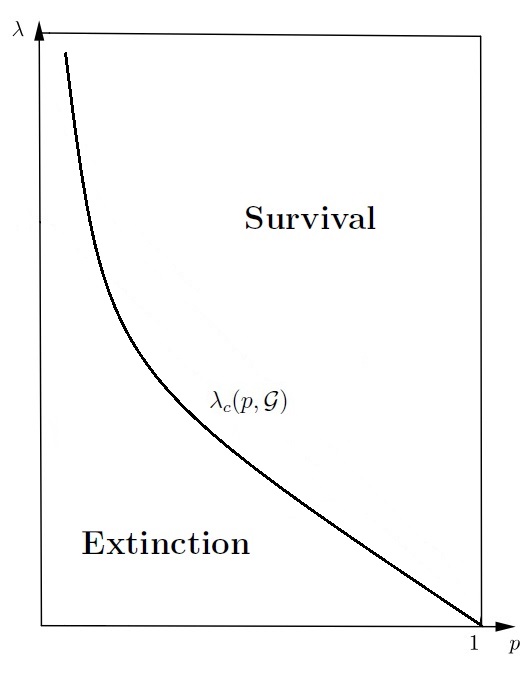}
\caption{Phase transition to $CC(\mathcal{G},\lambda,p),$ with $\mathcal{G}=\mathbb{Z}^d \ (\text{or} \ \mathbb{T}^d).$ }
\label{F: TransicaoAnt}
\end{figure}

\section{Non spatial models}

So called Catastrophe Models have been studied extensively and are quite close to our model, see Kapodistria et al.~\cite{Kapodistria} for references on the subject. Particularly relevant is the birth and death process with binomial catastrophes, see Example 2 in Brockwell~\cite{Brockwell}. We now describe this model. It is a  single colony model. Each individual gives birth at rate $\lambda>0$ and dies at rate  $\mu>0$. Moreover,  catastrophes (i.e. collapses) happen at rate $a>0$. When a catastrophe happens every individual in the colony has a probability $p$ of surviving and $1-p$ of dying, independently of each other. Brockwell~\cite{Brockwell} has shown that survival  (i.e. at all times there is at least one individual in the colony)
has positive probability if and only if $\lambda>\mu$ and
$$\mu-a\log p<\lambda.$$
Hence, there is a critical value for $p$,
$$p_1=\exp(-\frac{\lambda-\mu}{a}).$$
The single colony model survives if and only if $p>p_1$. \\

Next we introduce a non spatial version of our model and compare it to the catastrophe model above. Consider a model for which every individual gives
birth at rate $\lambda$ and dies at rate $\mu$. We start with a single individual and hence with a single colony. 
When a colony collapses  individuals in the colony survive with probability $p$ and die with probability $1-p$ independently of each other. 
Every surviving individual founds a new colony which eventually collapses.  Colonies collapse independently of each other at rate $a>0$. 
The proof in Schinazi~\cite{S2014} may be adapted to show that survival has positive probability if and only if
$$pE\left [\exp\left ((\lambda-\mu)T\right )\right ]>1,$$
where $T$ has a rate $a$ exponential distribution. It is easy to see that if $\lambda\geq \mu+a$ then the expected value on the l.h.s. is $+\infty$ and the inequality
holds for any $p>0$. It is also easy to see that the inequality cannot hold if $\lambda\leq \mu$. Hence, from now on we assume that $\mu<\lambda<\mu+a.$
After computing the expected value and solving for $p$ we get that survival is possible if and only if 
$$p>1-\frac{\lambda-\mu}{a}.$$
That is, when $\mu<\lambda<\mu+a$ the  model with multiple colonies has a critical value
$$p_2=1-\frac{\lambda-\mu}{a}.$$
The multiple colonies model survives if and only if $p>p_2$.

Since $\exp(-x)> 1-x$ for all $x\not = 0$ we have that
$p_1>p_2$ for any $\lambda>\mu$. Hence, it is easier for the model with multiple colonies to survive than it is for 
the model with a single colony. That is,  living in multiple smaller colonies is a better survival strategy than living in a single big colony. 
Note that this conclusion was not obvious. The one colony model has a 
catastrophe rate of $a$ while the multiple colonies model has a catastrophe rate of $na$ if
there are $n$ colonies. Moreover, a catastrophe is more likely to wipe out a smaller colony than a larger one. On the other hand 
multiple colonies give multiple chances for survival and this turns out to be a critical advantage of the multiple colonies model
over the single colony model.

\section{proofs}

\subsection{Auxiliary results}

For $\eta_t$ being a $CC(\mathcal{G},\lambda,p)$ process, we know that some colonization attempts  will not succeed because the vertex on which the attempted colonization takes place is already occupied. 
This creates dependence between the number of new colonies created upon the collapse of different colonies. Because of this lack of independence, explicit probability computation seems impossible. In order to prove Theorem \ref{extinction}, we introduce a branching-like process
which dominates  $\eta_t$, in a certain sense, and for which explicit computations are possible. This process is denoted by $\xi_t$ and defined as follows.\\ 

\noindent
\textbf{Auxiliary process $\xi_t$ :}\\
Each vertex of $\mathcal{G}$ might be empty or occupied by a number of colonies. Each colony
starts from a single individual. The number of individuals in a colony at time $t$ is determined by a pure
birth process of rate $\lambda$. Each colony is associated to a mean 1 exponential random variable.
When the exponential clock rings for a colony it collapses and each individual, independently from everything
else, survives with probability $p>0$ or dies with probability $1-p$. Each individual who survives tries to create a new colony at one of the nearest neighbor vertices picked at random. 
At every neighboring vertex we allow at most one new colony to be created. 
Hence,  in the process $\xi_t$ when a colony placed at vertex $x$ collapses, it is replaced by 0,1,.., or degree$(x)$ new colonies, each new colony on a distinct neighboring site of $x$. \\

Observe that birth and collapse rates are the same for colonies in $\xi_t$ and $\eta_t$.  To each colony created in process $\eta_t$ corresponds a colony created in the process $\xi_t$. But not every colony created in the process $\xi_t$ has its correspondent in the process $\eta_t$. 
Techniques such as in Liggett \cite[Theorem 1.5 in chapter III]{Liggett} 
can be used to construct the processes $\xi_t$ and $\eta_t$  in the same probability space in such a
way that, if they start with the same initial configuration, if there is a colony of size $i$ on a vertex $x$ for
$\eta_t$ then there is at least one colony of size $i$ for $\xi_t$ on the same vertex $x$.

\begin{lem}
\label{L: lemaaux1} Let $\mathcal{G}$ be a $m-$regular graph and $W_m(\lambda,p)$ the number of new colonies created by individuals of a collapsing colony in the process $\xi_t$ on $\mathcal{G}$. Then 
\begin{itemize}
\item[$(i)$] $\mu(m):=\mathbb{E}[W_m(\lambda,p)] = m-\frac{m}{\lambda}\sum_{k\geq 1} B\left(1+\frac{1}{\lambda},k\right)\left(1-\frac{p}{m}\right)^k.$
\item[$(ii)$] $q_{\lambda}:=\mathbb{P}[W_m(\lambda,p)=m]\rightarrow 1$ as $\lambda \rightarrow \infty.$
\end{itemize}
\end{lem} 

\begin{obs}
\label{O: obs1} 
Observe that for the process $\eta_t$ the probability that upon a collapse at vertex $x$ each one of its $m$ neighbors gets at least one colonization attempt is equal to  $q_{\lambda}$.
\end{obs}

\begin{proof}[Proof of Lemma \ref{L: lemaaux1}]  $(i)$ Consider a colony at some vertex $x$ of $\mathcal{G}$. Let $Y$ be the number of individuals in the colony at collapse time. Then 

\begin{eqnarray}\label{E1: lemaaux1}
\mathbb{P}[Y=k]=\int_0^{\infty} e^{-t} e^{-\lambda t}(1-e^{-\lambda t})^{k-1}dt =\frac{1}{\lambda}B\left(1+\frac{1}{\lambda},k\right),
\end{eqnarray}
where the last equality is obtained by the substitution $u=e^{-\lambda t}$ and the definition of the beta function.\\

Enumerate  each neighbor of vertex $x$ from 1 to $m$, then  
$W_m(\lambda,p)=\sum_{i=1}^m I_i,$ where $I_i$ is the indicator function of the event \{A new colony is created in the $i-$th neighbor of $x$.\}. Hence, 
\begin{eqnarray}\label{E2: lemaaux1}
\mathbb{E}[W_m(\lambda,p)]=\sum_{i=1}^m \mathbb{P}[I_i=1]=m\mathbb{P}[I_1=1].
\end{eqnarray}

Observe that
\begin{eqnarray}
\mathbb{P}[I_1=1|Y=k]
=1-\left(1-\frac{p}{m}\right)^k.\nonumber
\end{eqnarray} 

Therefore, 
\begin{eqnarray}\label{E3: lemaaux1}
\mathbb{P}[I_1=1]&=&\sum_{k=1}^\infty\left[1-\left(1-\frac{p}{m}\right)^k\right]\mathbb{P}[Y=k] \nonumber \\
&=&1-\frac{1}{\lambda} \sum_{k=1}^\infty \left(1-\frac{p}{m}\right)^k B\left(1+\frac{1}{\lambda},k\right) 
\end{eqnarray}
where the last equality is obtained by (\ref{E1: lemaaux1}). Substituting (\ref{E3: lemaaux1}) in (\ref{E2: lemaaux1}) we obtain the desired result.\\

$(ii)$ Observe that
\begin{eqnarray}\label{E4: lemaaux1}
\mathbb{P}[W_m(\lambda,p)=m]&=&\mathbb{P}[I_1=1,\ldots,I_m=1]\nonumber\\
&=&1-\mathbb{P}[I_i=0 \text{ for some } i \in {1,\ldots m}] \nonumber\\
&\geq &1-\sum_{i=1}^m \mathbb{P}[I_i=0] \nonumber\\
&=&1-m \mathbb{P}[I_1=0] \nonumber\\
&=&1-\frac{m}{\lambda} \sum_{k=1}^\infty \left(1-\frac{p}{m}\right)^k B\left(1+\frac{1}{\lambda},k\right) \nonumber\\
&\geq &1-\frac{m}{\lambda} \sum_{k=1}^\infty \left(1-\frac{p}{m}\right)^k B\left(1,k\right) 
\end{eqnarray}
Letting $\lambda\rightarrow \infty$ in (\ref{E4: lemaaux1}) we obtain the result.

\end{proof}

\subsection{Proofs of main results}

\begin{proof}[Proof of Theorem \ref{extinction} (i)]

Consider $\xi_t$ starting with one colony at the origin and let $Z_0=1.$ This colony we
call the $0$-th generation.
Upon collapse of that colony a random number of new colonies are created. Denote this 
random number by $Z_1$. These are the first generation colonies. Every first generation colony
gives birth (at different random times) to a random number of new colonies. These new colonies
are the second generation colonies and their total number is denoted by $Z_2$.
More generally, let $n \geq 1$, if $Z_{n-1} = 0$ then $Z_n=0$, if
$Z_{n-1} \geq 1$ then $Z_n$ is the total number of colonies created by the previous generation colonies.

We claim that  $Z_n$, $n=0,1,\dots$ is a Galton-Watson process. This is so because the offsprings
of different colonies in the process $\xi_t$ have the same distribution and are independent.

The process $Z_n$ dies out if and only if $\mathbb{E}[Z_1]\leq 1$.
From Lemma~\ref{L: lemaaux1}$.(i)$ we know that $\mathbb{E}[Z_1]=\mu(m).$ 

Observe that if the process $Z_n$ dies out, the same happens to $\xi_t$ and hence to $\eta_t$.
It is easy to see that the proof works if the process starts from any finite number of colonies.
\end{proof}
  
\begin{proof}[Proof of Theorem \ref{extinction} (ii)]

From the monotonic property of $\eta_t$ it suffices to show local extinction for the process
starting with one individual at each vertex of $\mathbb{Z}^d$. We will actually prove
local extinction for the process $\xi_t$ starting with one individual at each vertex of $\mathbb{Z}^d.$

Fix a vertex $x$. If at time $t$ there exists a colony at vertex $x$ (for the process $\xi_t$)
then it must descend from a colony present at time $0$. Assume that the colony at $x$
descends from a colony at some site $y$.   Let $Z_n(y)$ be the number of colonies at the
$n$-th generation of the colony that started at $y$. The process $Z_n(y)$ has the same distribution 
as the process $Z_n$ defined above. In order for a descendent of $y$ to eventually reach $x$ the
process $Z_n(y)$ must have survived for at least $d(x,y)$ generations.
This is so because each generation gives birth only on nearest neighbors vertices. The process
$Z_n(y)$ is a Galton-Watson process with $Z_0(y)=1$ and mean offspring $\mu=\mu(2d).$

Let $n=d(x,y)$, then
\[\mathbb{P}(Z_n(y)\geq 1)\leq \mathbb{E}(Z_n(y))=\mu^n=\mu^{d(x,y)}\]
and
\[ \sum_{y\in\mathbb{Z}^d}\mathbb{P}(Z_n(y)\geq 1)\leq \sum_{y\in\mathbb{Z}^d}\mu^{d(x,y)} = \]
\[ \sum_{n \ge 1} \#\{x \in \mathbb{Z}^d: d(x,y)=n\} \ \mu^n  = \sum_{n \ge 1} {n+d-1 \choose n} \mu^n <\infty,\] 
for $\mu<1$.

The Borel-Cantelli lemma shows that almost surely there are only finitely many $y$'s such that descendents
from $y$ eventually reach $x$. From $(i)$ we know that a process starting from
a finite number of individuals dies out almost surely. Hence, after a finite random time there will be no colony at vertex $x$.
\end{proof}

\begin{proof}[Proof of Theorem \ref{extinction} (iii)]

The proof is analogous to $(ii)$. In this case, $\mu=\mu(d+1)$ and 

\[ \sum_{y\in\mathbb{T}^d} \mu^{d(x,y)}=\sum_{n \ge 1} (d+1) d^{n-1}\mu^n <\infty, \]
for  $\mu<1/d$.
\end{proof}

\begin{proof}[Proof of Theorem \ref{T: survival}:]

We first give the proof on the one dimensional lattice $\mathbb{Z}$.

We start by giving an informal construction 
of the process. We put a Poisson process with rate 1 at every site of $\mathbb{Z}$.
All the Poisson processes are independent. At the Poisson process jump times the colony
at the site collapses if there is a colony. If not, nothing happens. We start the process with finitely
many colonies. Each colony starts with a single individual and is associated to a Yule process with
birth rate $\la$. At collapse time, given that the colony at site $x$ has $n$ individuals we have a
binomial random variable with parameters $(n,p)$. The binomial gives the number $k$ of
potential survivors. If $k\geq 1$ then each survivor attempts to found a new colony on $x+1$
with probability $\frac{1}{2}$ or on $x-1$ also with  probability $\frac{1}{2}$.
The attempt succeeds on an empty site. We associate a new Yule process to a new colony,
independently of everything else.

We use the block construction presented in Bramson and Durrett~\cite{B&D}. First some notation.
For integers $m$ and $n$ we define

$$I=[-L,L] \qquad I_m=2mL+I,$$

$$B=(-4L,4L)\times [0,T]\qquad B_{m,n}=(2mL,nT)+B,$$
where
$$T=\frac{5}{2} L$$
and
$$\mathbb {L}=\left\{ (m,n)\in \mathbb{Z}^2: m+n\mbox{ is even}\right \}.$$

We say that the interval $I_m$ is {\sl half-full} if at least every other site of
$I_m$ is occupied. That is, the gap between two occupied sites in $I_m$ is at most one site.

We declare $(m,n)\in \mathbb{L}$ wet if starting with $I_m$ half-full at time $nT$
then $I_{m-1}$ and $I_{m+1}$ are also half-full at time $(n+1)T$. Moreover, we want
the last event to happen using only the Poisson and Yule processes inside the box $B_{m,n}$.
That is, we consider the process {\sl restricted} to $B_{m,n}$.

We are going to show that for any $\epsilon>0$ there are $\la$, $L$ and $T$ so that for any $(m,n)\in  \mathbb{L}$
$$P((m,n)\mbox{ is wet})\geq 1-\epsilon.$$

By translation invariance it is enough to prove this for $(m,n)=(0,0)$. The proof has two steps.

$\bullet$ Let $E$ be the event that every collapse
in the finite space-time box $B$ is followed by at least one attempted colonization on the left and one on the right of
 the collapsed site. We claim that for every $\ee>0$, we can pick $L$, $T$ and  $\la>0$ large enough so that $P(E)\geq 1-\ee.$
We now give the outline of why this is true. Since collapse times are given by rate one Poisson processes on each site
of $B$ the total number of collapses inside $B$ is bounded above by a Poisson distribution with rate $(8L+1)T$. Hence, with
high probability there are less than $2(8L+1)T$ collapses inside $B$ for $L$ large enough. We also take $\la$ large enough so that at every collapsing time
the colony will have so many individuals that attempted colonizations to the left and right will be almost certain (see Lemma~\ref{L: lemaaux1}$.(ii)$). Since the number of collapses can be
bounded with high probability the probability of the event $E$ can be made arbitrarily close to 1.

$\bullet$ At time 0 we start the process with the interval $I$ half-full. Let $r_t$ and $\ell_t$ be respectively
the leftmost and rightmost occupied sites at time $t\geq 0$. Conditioned on the event $E$ it is easy to see
that the interval $[\ell_t, r_t]$ is half-full at any time $t\leq T$. Observe also that conditioned on $E$, every time there is a collapse 
at $r_t$ then $r_t$ jumps to $r_t+1$. Since the number of collapses at $r_t$ is a Poisson process with rate 1 we have that
$\frac{r_t}{t}$ converges to 1. Hence, for $T=\frac{5}{2} L$ we have that $r_T$ belongs to $(3L,4L)$ with a probability arbitrarily
close to 1 provided $L$ is large enough. A symmetric argument shows that $\ell_T$ belongs to $(-4L,-3L)$. Since the interval
$[\ell_T,r_T]$ contains $I_{-1}$ and $I_1$, both of these intervals are half-full. Hence, for any $\ee>0$ we can pick $L$ and $\la$ large enough
so that $P((0,0))\mbox{ is wet})\geq 1-\epsilon.$

The preceding construction gives a coupling between our colonization and collapse model and an oriented percolation model on $ \mathbb{L}$. The oriented percolation model is 1-dependent and it is well known that for $\ee>0$ small enough $(0,0)$ will be in an infinite wet cluster which contains infinitely many vertices like $(0,2n)$, see Durrett~\cite{DURRETT1984}. That fact corresponds, by the coupling, to local survival in the colonization and collapse model.
Note that the proof was done for the process restricted to the boxes $B_{m,n}$. However, if
this model survives then so does the unrestricted model. This is so because the model is attractive and more births can only help survival. \\

Consider now $\mathbb{Z}^d$ with $d \ge 2$. 
First observe that from the case $d=1$ we have a sufficient condition for local survival for the process
restricted to a fixed line
of $\mathbb{Z}^d$. Since the model is attractive this is enough to show local survival on $\mathbb{Z}^d$.
This is so because the unrestricted model will bring more births (but not more deaths) to the sites on
the fixed line.

For  $\mathbb{T}^d$ the  local survival follows analogously as in $\mathbb{Z}^d$, observing that 
$\mathbb{Z}$ is embedded in $\mathbb{T}^d$.
\end{proof}

\begin{obs}
Our argument shows that both critical values, $\lambda_c(p,\mathbb{Z}^d)$ and $\lambda_c(p,\mathbb{T}^d)$, decrease with $d$. The more difficult issue is whether the critical value is strictly decreasing. We conjecture it is but this is a hard question even for the contact process, see Liggett~\cite{Liggett2}.
\end{obs}


\begin{thebibliography}{99}


\bibitem{B&D}
{M. Bramson and R. Durrett}
Simple proof of the stability criterion of Gray and Griffeath. 
\textit{Probability Theory and Related Fields} 
\textbf{80}, 293-298. (1988). 

\bibitem{Brockwell} {P.J. Brockwell.} 
The extinction time for a general birth and death process with catastrophes.
\textit{Journal of Applied Probability} \textbf{23}, 851-858. (1986).

\bibitem{DURRETT1984} {R. Durrett.}  
Oriented Percolation in two dimensions. 
\textit{The Annals of Probability.} \textbf{12} (4), 999-1040. (1984).

\bibitem{Hanski} {I. Hanski}
Metapopulation Ecology.
\textit{Oxford Univ. Press, Oxford.} (1999).

\bibitem{Kapodistria} {S. Kapodistria, T. Phung-Duc and J. Resing.} 
Linear birth/immigration-death process with binomial catastrophes.
\textit{Probability in the Engineering and Informational Sciences} \textbf{30} (1), 79-111 (2016).

\bibitem{Liggett} {T. Liggett.}
Interacting particle systems. 
\textit{Springer-Verlag.} (1985).

\bibitem{Liggett2} {T. Liggett.}
Stochastic Interacting Systems: Contact, Voter and Exclusion Processes.
 \textit{Springer, New York.} (1999).


\bibitem{Luke}{Y.L. Luke.} 
 The Special Functions and Their Approximations.
 \textit{Academic Press, New York.} vol. 1. (1969)

\bibitem{S2014} {R. Schinazi.} 
Does random dispersion help survival?
\textit{Journal of Statistical Physics}, \textbf{159} (1),101-107. (2015).












\end{thebibliography}
\end{document}